\documentclass[12pt]{article}
\usepackage{amsmath, amscd, amssymb, latexsym, epsfig, color, amsthm}
\newtheorem{theorem}{Theorem}[section]

\newtheorem{corollary}[theorem]{Corollary}

\newtheorem{lemma}[theorem]{Lemma}

\theoremstyle{definition}
\newtheorem{defn}[theorem]{Definition}

\newtheorem{remark}[theorem]{Remark}

\newenvironment{proofof*}[1]
{\begin{trivlist}\item {\it Proof of {#1}.}}{\end{trivlist}}

\def\Z{\mathbb{Z}}
\def\m{\mathfrak{m}}

\newcommand\lk{\mbox{\upshape lk}\,}
\newcommand\inc{\iota}
\newcommand\field{\Bbbk}
\newcommand{\soc}{\mbox{\upshape Soc}\,}
\newcommand{\bK}{\mathcal{K}}
\newcommand{\bC}{\mathcal{C}}

\newcommand\kk{\Bbbk}
\newcommand\mm{\mathfrak{m}}

\DeclareMathOperator\tor{\mathrm{Tor}}

\DeclareMathOperator\cost{\mathrm{cost}}
\DeclareMathOperator\coker{\mathrm{coker}}
\DeclareMathOperator\id{\mathrm{id}}

\begin{document}

\title{\vspace{-1ex}Face numbers of pseudomanifolds with isolated
singularities}

\author{
Isabella Novik
\thanks{Research partially supported by Alfred P.~Sloan Research
Fellowship and NSF grant DMS-0801152}\\
\small Department of Mathematics, Box 354350\\[-0.8ex]
\small University of Washington, Seattle, WA 98195-4350, USA\\[-0.8ex]
\small \texttt{novik@math.washington.edu}
\and Ed Swartz
\thanks{Research partially supported by NSF grant DMS-0900912}\\
\small Department of Mathematics, \\[-0.8ex]
\small Cornell University, Ithaca NY, 14853-4201, USA\\[-0.8ex]
\small \texttt{ebs22@cornell.edu}
\date{April 28, 2010}
}

\maketitle

\begin{abstract}
We investigate the face numbers of simplicial complexes with Buchsbaum
vertex links, especially pseudomanifolds with isolated singularities.
This includes deriving Dehn-Sommerville relations for pseudomanifolds
with isolated singularities and establishing lower bound theorems  when
the singularities are also homologically isolated.  We give formulas for
the Hilbert function of a generic Artinian reduction of the face ring
when the singularities are homologically isolated and for any pure
two-dimensional complex.  Some examples of spaces where the
$f$-vector can be completely characterized are described.

Let $\Delta$ and $\Delta'$ be two simplicial complexes that
are homeomorphic and have the same $f$-vector.
In \cite{MNS} the question is raised of whether or not the Hilbert
functions of generic Artinian reductions of their face rings are
identical. We prove that this is the case if the spaces have isolated
singularities and are PL-homeomorphic. \end{abstract}

\noindent{\it 2010 Mathematics Subject Classification.\hspace{-1.2pt}}
05E45,~13F55,~05E40,~13D45

\section{Introduction}

Stanley's introduction of commutative algebra into the study of face
numbers of simplicial complexes via the face ring  revolutionized the
subject.   For instance, he was able to give a complete characterization
of all possible $f$-vectors of Cohen-Macaulay complexes \cite{St77}.  A
simplicial complex is Cohen-Macaulay (CM) if the homology of the link of
any face (including the empty face) is trivial except possibly in the top
dimension.  In \cite{Sch} Schenzel showed how to extend many of these
ideas to Buchsbaum complexes, that is, pure complexes
whose vertex links are CM. This
includes triangulations of manifolds (with and without boundary).  In
\cite{MNS} algebraic aspects of  face rings of spaces with more
complicated singularities were studied.  Here we consider combinatorial
and algebraic properties of face rings of pure complexes whose vertex
links
are Buchsbaum with a particular emphasis on pseudomanifolds with isolated
singularities.

We start by introducing the basic notations, definitions and results we
need for simplicial complexes, face rings and local cohomology.  In
Section \ref{Dehn-Sommerville} we derive  Dehn-Sommerville relations for
pseudomanifolds with isolated singularities.  This is followed by a
close examination of spaces with homologically isolated singularities.
(See Section \ref{homologically isolated singularities} for a
definition.)  In this case we are able to establish formulas for the
Hilbert function of an arbitrary Artinian reduction of the face ring and
a
lower bound formula analogous to \cite[Theorem 5.2]{NS-socle} for
manifolds.  Highly connected (in the topological sense) spaces are the
subject of Section~\ref{depth}.  A corollary of these results is a
formula for the Hilbert function of generic Artinian reductions of  pure
two-dimensional complexes (Eq.~(\ref{2d})). Some examples of
pseudomanifolds where we can give a
complete characterization of their $f$-vectors are the subject of Section
\ref{Examples}.  Finally, we show that the Hilbert function of a generic
Artinian reduction of the face ring of a complex with isolated
singularities is determined by its $f$-vector and PL-homeomorphism
type.

\section{Preliminaries}

In this section we review certain aspects of simplicial complexes
and their face rings that will be needed in the rest of the paper.
An excellent general reference to this material is \cite{St96}.

\subsection{Simplicial complexes}
Let $\Delta$ be a $(d-1)$-dimensional simplicial complex. The main object
of our study is the {\em $f$-vector} of $\Delta$,
$f(\Delta):=(f_{-1}(\Delta), f_0(\Delta),\ldots, f_{d-1}(\Delta))$,
where $f_i(\Delta)$ denotes the number of $i$-dimensional faces of
$\Delta$. It is sometimes more convenient to work with the $h$-vector of
$\Delta$,
$h(\Delta)=(h_0(\Delta), \ldots, h_d(\Delta))$, defined by $$
 \sum_{i=0}^d h_i x^{d-i} = \sum_{i=0}^d f_{i-1} (x-1)^{d-i}.
$$
Thus, $h_0(\Delta)=1$, $h_1(\Delta)=f_0(\Delta)-d$,
 and $h_{d}(\Delta)=(-1)^{d-1}\tilde{\chi}(\Delta)$,
where $\tilde{\chi}(\Delta)$ is the reduced Euler characteristic of
$\Delta$.

Fix an infinite field $\kk$.
Another set of invariants we need are the {\em Betti
numbers}, $$\beta_i(\Delta):=\dim_{\kk}\tilde{H}_i(\Delta;\kk), \quad
0\leq i \leq \dim\Delta.$$
Here $\tilde{H}_i(\Delta;\kk)$ is the $i$-th reduced
simplicial homology of $\Delta$ computed with coefficients in $\kk$.

For a face $F$ of $\Delta$, the {\em link} of $F$ in $\Delta$ is $\lk F
:=\{G : F \cap G = \emptyset,\  F \cup G \in \Delta\}.$
We use the following terminology: $\Delta$ is a {\em
$\kk$-Cohen--Macaulay complex}
 ($\kk$-CM, for short) if for every face $F\in\Delta$ (including the
empty face), $\beta_i(\lk F)=0$ for all $i<d-|F|-1$;
if in addition $\beta_{d-|F|-1}(\lk F)=1$ for all $F\in\Delta$, then
$\Delta$ is called a {\em $\kk$-homology sphere}.  When no confusion is
likely we may suppress the field $\field$.

A face of $\Delta$ is a {\em facet} if it is maximal under inclusion. We
say that $\Delta$ is {\em pure} if all facets of $\Delta$ have dimension
$d-1$, and that $\Delta$ is a {\em pseudomanifold} if it is  pure and
each codimension-one face of $\Delta$ is contained in exactly two facets.

A pure simplicial complex $\Delta$ is {\em $\kk$-Buchsbaum} if all vertex
links of $\Delta$ are $\kk$-CM, and it is a {\em $\kk$-homology manifold}
if all vertex links are $\kk$-homology spheres. A pure simplicial complex
$\Delta$
is a {\em $\kk$-space with isolated singularities} if all vertex links of
$\Delta$
are $\kk$-Buchsbaum (equivalently, all edge links are $\kk$-CM), and it
is a {\em $\kk$-pseudomanifold with isolated singularities} if all vertex
links are $\kk$-homology manifolds (equivalently, all edge links are
$\kk$-homology spheres). A vertex $v$ of a space (pseudomanifold, resp.)
with isolated
singularities is called {\em singular} if the link of $v$ is not $\kk$-CM
(not a $\kk$-homology sphere, resp.).

\subsection{Face rings}
The {\em face ring} (also known as the {\em Stanley-Reisner ring}) of a
simplicial complex $\Delta$ on the vertex set $V$ is
$$\kk[\Delta]:=\kk[x_v  :  v\in V]/I_\Delta,$$
where $I_\Delta$ is the squarefree monomial ideal generated by non-faces:
$$I_\Delta = (x_{v_1}\ldots x_{v_k}  :  \{v_1,\ldots,
v_k\}\notin\Delta).
$$
It is a result of Reisner \cite{Reis} that $\Delta$ is a $\kk$-CM
simplicial complex if and only if $\kk[\Delta]$ is a Cohen--Macaulay
ring, and it is a result of Schenzel \cite{Sch}
that $\Delta$ is a $\kk$-Buchsbaum complex
if and only if $\kk[\Delta]$ is a Buchsbaum ring.

A {\em linear system of parameters} (an {\em l.s.o.p.}, for short)
for $\kk[\Delta]$ is a set of $d=\dim\Delta+1$ linear
forms  $\Theta=\{\theta_1, \ldots, \theta_d\}$ such that
$$\kk(\Delta, \Theta):= k[\Delta]/(\theta_1, \ldots, \theta_d)$$
is a finite-dimensional $\kk$-space. If $\kk$ is an infinite field, then
an l.s.o.p.~always exists. The ring $\kk(\Delta, \Theta)$ is called an
{\em Artinian reduction} of $\kk[\Delta]$.

\subsection{Local cohomology}   \label{local-coh:subsection}
Let $\m=(x_v  :  v\in V)$ be the irrelevant ideal of the polynomial
ring $S:=\kk[x_v  :  v\in V]$. For a finitely-generated
graded  $S$-module $M$,
we denote by $M_k$ the $k$-th homogeneous component of $M$
and by $H^i_\m(M)$ the $i$-th local cohomology of $M$ with respect to
$\m$.
For $t\in S$ define $(0:_M t) :=\{\nu\in M : t\nu=0\}$.
The {\it socle} of $M$ is $$\soc(M)=\bigcap_{v\in V}(0:_M x_v)=\{\nu\in
M  :  \m\nu=0\}.$$

For a simplicial complex $\Delta$, Gr\"abe \cite{Gr} gave a description
of $H^i_\m(\kk[\Delta])$ and its $S$-module structure in terms of the
simplicial cohomology of the links of $\Delta$ and the maps between them.
In the case when $\Delta$ is a space with isolated singularities, this
description takes the following simple form that
we review now. For $v\in V$, let {\em costar} of $v$ be $\cost(v)=\{F \in
\Delta : v\notin F\}$, and consider the $i$-th simplicial cohomology of
the pair $(\Delta, \cost v)$
with coefficients in $\kk$: $$H^i_{\{v\}}(\Delta):=H^i(\Delta, \cost
v;\kk) \cong \tilde{H}^{i-1}(\lk v; \kk).$$
Also set $H^i_{\emptyset}(\Delta):=H^i(\Delta, \emptyset;\kk)
 =\tilde{H}^i(\Delta;\kk),$
and let $\inc^\ast: H^i_{\{v\}}(\Delta)\to H^i_\emptyset(\Delta)$, for
$v\in V$, be the map induced by inclusion. In the following,
$\delta_{wv}$ denotes the Kronecker delta, and $(s_v)_{v\in V}$ an
element of $\bigoplus_{v\in V} H^{i}_{\{v\}}(\Delta)$.

\begin{theorem} {\rm [Gr\"abe]}  \label{Grabe}
Let $\Delta$ be a space with isolated singularities and $-1\leq i< \dim
\Delta$. Then $$H^{i+1}_\m(\kk[\Delta])_{-j}=
\left\{\begin{array}{ll}
 0 & \mbox{ (if $j<0$)}, \\
 H^{i}_\emptyset(\Delta) & \mbox{ (if $j=0$)},\\
 \bigoplus_{v\in V} H^{i}_{\{v\}}(\Delta) & \mbox{ (if $j>0$)}.
 \end{array}
 \right.
$$
The $S$-module structure on $H^{i+1}_\m(\kk[\Delta])$ is given by
\begin{eqnarray*}
\cdot x_w :
 H^{i+1}_\m(\kk[\Delta])_{-(j+1)} &\to& H^{i+1}_\m(\kk[\Delta])_{-j} \\
          (s_v)_{v\in V} &\mapsto&  \left\{
           \begin{array}{ll}
                   \inc^{\ast}(s_w) & \mbox{ if $j=0$,}\\
                  (\delta_{wv}s_v)_{v\in V}  & \mbox{ if $j>0$.}
             \end{array}
       \right. \end{eqnarray*} \end{theorem}

\section{Dehn-Sommerville relations} \label{Dehn-Sommerville}

One of the nicest properties of the $h$-vectors of homology manifolds
is the Dehn-Sommerville relations due to Klee \cite{Kl64} asserting that
if $\Gamma$ is an $(r-1)$-dimensional homology manifold, then
\begin{equation}   \label{DS}
    h_{r-i}(\Gamma)=h_i(\Gamma)+ (-1)^{i-1}\binom{r}{i}\left(1+(-1)^r
\tilde{\chi}(\Gamma)\right)
 \quad \mbox{ for all } 0\leq i \leq r. \end{equation}
The goal of this section is to
derive the following generalization of Eq.~(\ref{DS}) for pseudomanifolds
with isolated singularities.  In principle this could be done using the
results of \cite{CY} or \cite[Lemma 2]{Wall} that are stated in terms of
$f$-vectors and are
designed to include more general complexes. However,  as we will require
an $h$-vector version and only need it for pseudomanifolds with isolated
singularities, we derive the formula ourselves.
By convention, for all $a, b\in\Z$, $\binom{b}{0}=1$, and
$\binom{b}{a}=0$ if $a<0$.

\begin{theorem}  \label{gen-DS}
Let $\Delta$ be a $(d-1)$-dimensional simplicial complex on the
vertex set $V$. If $\Delta$ is a pseudomanifold with isolated
singularities, then for all $0\leq i \leq d$,
$$
h_{d-i}(\Delta)=h_i(\Delta) +
(-1)^{i-1}\binom{d}{i}\left(1+(-1)^d\tilde{\chi}(\Delta)\right)
+(-1)^i\binom{d-1}{i-1}
 \sum_{v\in V}\left(1+(-1)^{d-1}\tilde{\chi}(\lk v)\right).
$$
\end{theorem}

\begin{proof}  We utilize Eq.~(\ref{DS}) applied to the links of
vertices --- notice that since $\Delta$ is a pseudomanifold
with isolated singularities, all vertex links are $(d-2)$-dimensional
homology manifolds.
Another ingredient needed for the proof is the following connection
between the $h$-vector of a pure simplicial complex and the sums of the
$h$-vectors of the vertex links (see \cite[Proposition 2.3]{Sw04}):
\begin{equation} \label{short-h}
ih_i(\Delta) + (d-i+1)h_{i-1}(\Delta)=
\sum_{v\in V}h_{i-1}(\lk v) \quad \mbox{ for all }
 1\leq i \leq d.
\end{equation}

The proof of the theorem is by induction on $i$. For $i=0$ the statement
reduces to $$
h_d(\Delta)-h_0(\Delta)=-\left(1+(-1)^d\tilde{\chi}(\Delta)\right).
$$
This holds as
$h_0(\Delta)=1$ and $h_d(\Delta)=(-1)^{d-1}\tilde{\chi}(\Delta)$
for all simplicial complexes.

For $i>0$, we have
\begin{eqnarray*}
ih_i(\Delta) &+& (d-i+1)h_{i-1}(\Delta) \\
 &=& \sum_{v\in V} h_{i-1}(\lk v)\\
 &=& \sum_{v\in V} \left[h_{(d-1)-(i-1)}(\lk v)+
 (-1)^{i-1}\binom{d-1}{i-1}
 \left(1+(-1)^{d-1}\tilde{\chi}(\lk v)\right)\right]\\
 &=& \left[\sum_{v\in V} h_{d-i}(\lk v)\right]+
 (-1)^{i-1}\binom{d-1}{i-1}
 \sum_{v\in V} \left(1+(-1)^{d-1}\tilde{\chi}(\lk v)\right)\\
&=& (d-i+1)h_{d-i+1}(\Delta)+ih_{d-i}(\Delta)+
 (-1)^{i-1}\binom{d-1}{i-1}\sum_{v\in V}
 \left(1+(-1)^{d-1}\tilde{\chi}(\lk v)\right),
\end{eqnarray*}
where the first and the last steps follow from Eq.~(\ref{short-h}), and
the second step from Eq.~(\ref{DS}) applied to the vertex links. Thus
\begin{equation}  \label{(i-1)->i}
i(h_{d-i}-h_i)= -(d-i+1)(h_{d-(i-1)}-h_{i-1})+
(-1)^i \binom{d-1}{i-1}\sum_{v\in V}
 \left(1+(-1)^{d-1}\tilde{\chi}(\lk v)\right).
\end{equation}
Substituting in Eq.~(\ref{(i-1)->i}) our inductive hypothesis for
$h_{d-(i-1)}-h_{i-1}$ and using that $$
\frac{1}{i} \cdot (d-i+1)\binom{d}{i-1}=\binom{d}{i} \quad \mbox{and}
\quad
\frac{1}{i} \cdot
 \left[(d-i+1)\binom{d-1}{i-2} + \binom{d-1}{i-1}\right]
=\binom{d-1}{i-1}
$$
yields the result.
\end{proof}

\noindent  The above formula can be written exclusively in terms of
$\chi(\Delta).$  When $d$ is odd every vertex link is an odd-dimensional
homology manifold, so its Euler characteristic is zero and Theorem~3.1
reduces to Eq.~(\ref{DS}). When $d$ is even, there exists $\Delta'$ ---
an odd-dimensional manifold with boundary whose boundary consists of the
disjoint union of the vertex links of $\Delta$ --- such that $\Delta$ is
homeomorphic to $\Delta'$ with each boundary component coned off. The Euler characteristic of $\Delta'$ is one-half the Euler
characteristic of its boundary and hence equals one-half the sum of the
Euler characteristics of the  vertex links and, of course, the Euler
characteristic of each cone is one.  Hence,
$\chi(\Delta) = \frac{1}{2} \sum_{v \in V} \chi(\lk v) + |V| - \sum_{v\in
V} \chi(\lk v)$ and $$\sum_{v \in V} \left(1+(-1)^{d-1}\tilde{\chi}(\lk
v)\right) = 2|V| -
\sum_{v \in V} \chi(\lk v) = 2\chi(\Delta).$$

\section{Homologically isolated singularities}  \label{homologically isolated singularities}

 We now turn our attention to inequalities. In this section we
derive lower bounds on the face numbers of spaces with
{\em homologically isolated singularities} --- a certain subclass of
spaces with isolated singularities introduced in \cite[Section 4]{MNS}.
To achieve this goal we first compute the Hilbert series of Artinian
reductions of such complexes generalizing Schenzel's formula for
Buchsbaum complexes  \cite{Sch}, which in turn is a generalization of
Stanley's formula for CM complexes \cite{St77}.

We use the notation of Subsection \ref{local-coh:subsection}.
For a linear form $\theta=\sum_{v\in V} a_vx_v$ with all $a_v\neq 0$,
consider the maps
$\inc^\ast: H^i_{\{v\}}(\Delta) \to H^i_\emptyset(\Delta)$
and their weighted direct sum
$$f^{i,\theta}: \left[\bigoplus_{v\in V}H^i_{\{v\}}(\Delta)\right]
 \to H^i_\emptyset(\Delta), \quad
 f^{i,\theta}=\sum_{v\in V}
 a_v\inc^{\ast}
 \left[H^i_{\{v\}}(\Delta) \to H^i_\emptyset(\Delta)\right].
 $$
Define $\bK_i^\theta(\Delta):=\dim_\kk \ker f^{i,\theta}$ and
$\bC_i^\theta(\Delta):=\dim_\kk \coker f^{i,\theta}$.

\begin{defn} We say that
a $(d-1)$-dimensional simplicial complex $\Delta$ on the vertex
set $V$ has
{\em homologically isolated singularities} if
(i) it is a space with isolated singularities, and (ii) for all
$0\leq i \leq d-2$, the images $\inc^\ast: H^i_{\{v\}}(\Delta) \to
H^i_\emptyset(\Delta)$
for $v\in V$ are linearly independent subspaces of
$H^i_\emptyset(\Delta)$.  Equivalently, the kernel of the above
map $f^{i, \theta}$ decomposes as a direct sum of the kernels of the
summands, and in particular it is independent of a specific choice of
$\theta$ (as long as all coefficients of $\theta$ are non-vanishing);
in this case we write $\bK_i$ and $\bC_i$ instead of $\bK_i^\theta$
and $\bC_i^\theta$, resp.
\end{defn}

Evidently, any complex with only one isolated singularity has
homologically isolated singularities. Many other examples of such
complexes are described in \cite[Section 4]{MNS}.

Schenzel \cite{Sch} proved that if $\Delta$ is a $\kk$-Buchsbaum
complex of dimension $d-1$, then the Hilbert series
 of $\kk(\Delta, \Theta)$ is independent of the choice of an
l.s.o.p.~$\Theta$ and is given by
\begin{equation} \label{Sch-eq}
F(\kk(\Delta, \Theta), \lambda) =
\sum_{i=0}^d \left(h_i(\Delta)+
\binom{d}{i}\sum_{j=1}^{i-1}(-1)^{i-j-1}
 \beta_{j-1}(\Delta)\right)\lambda^i.
\end{equation}
If $\Delta$ is $\kk$-CM, this formula reduces to Stanley's result:
$F(\kk(\Delta, \Theta), \lambda) = \sum_{i=0}^d h_i\lambda^i$.
Here we extend Schenzel's result to all complexes with homologically
isolated singularities. The main point is that
the Hilbert series of an Artinian reduction of the face ring of
such a complex $\Delta$ is determined by the $f$-numbers and the
homeomorphism type of $\Delta$.

\begin{theorem}  \label{h'}
Let $\Delta$ be a $(d-1)$-dimensional simplicial complex with
homologically
isolated singularities. Then for any l.s.o.p.~$\Theta$ of $\kk[\Delta]$,
the Hilbert series of $\kk(\Delta, \Theta)$ is equal to $\sum_{i=0}^d
h'_i(\Delta)\lambda^i$, where (for $0\leq i \leq d$)
\begin{eqnarray*}
h'_i(\Delta)=h_i(\Delta) &+&
\binom{d}{i}\sum_{j=1}^{i-1}(-1)^{i-j-1}\beta_{j-1}(\Delta)\\
&+&\binom{d-1}{i}\left[\bK_{i-1}(\Delta)-
\sum_{v\in V} \sum_{j=1}^{i-2}(-1)^{i-j}\beta_{j-1}(\lk v)\right].
\end{eqnarray*}
\end{theorem}

The crucial part of the proof is the following
property of the face rings of complexes
with homologically isolated singularities that was observed in
\cite[Section 4]{MNS}.
\begin{lemma}  \label{Buchsbaum}
Let $\Delta$ be a $(d-1)$-dimensional
simplicial complex with isolated singularities and let
$\theta=\sum_{v\in V}a_vx_v$ be a $1$-form with all  coefficients
non-vanishing. Then
\begin{enumerate}
\item The Hilbert series of the $i$-th local cohomology of
$\kk[\Delta]/(\theta)$, $F(H^i_\m(\kk[\Delta]/(\theta)), \lambda)$,
equals $\beta_i(\Delta)\lambda +
 (\bK^\theta_i(\Delta)+\bC^\theta_{i-1}(\Delta))$ for all $i<d-1$.
\item If $\Delta$ has homologically isolated singularities,
then $\kk[\Delta]/(\theta)$ is a Buchsbaum $S$-module. \end{enumerate}
\end{lemma}

\begin{proof} (Sketch) The first part follows from Theorem \ref{Grabe}
together with the long exact sequence in local cohomology. Moreover, the
condition on homologically
isolated singularities implies that $\m H^i(\kk[\Delta]/(\theta))=0$.
The second part then follows from \cite[Prop.~I.3.10]{StVo}.
\end{proof}

A comprehensive reference to Buchsbaum rings and modules is \cite{StVo}.
Here we will need only a few standard properties that we record in the
following lemma.

\begin{lemma} \label{properties}
Let $\Delta$ be a $(d-1)$-dimensional simplicial complex with
homologically isolated singularities and let
$\theta_1, \ldots, \theta_d$ be an l.s.o.p.~for $\field[\Delta]$
with $\theta_1$ having non-vanishing coefficients. Then
for all $1\leq s\leq d-1$,
\begin{enumerate}
\item $\kk[\Delta]\{s\}:=\kk[\Delta]/(\theta_1, \ldots, \theta_s)$ is a
Buchsbaum $S$-module  of Krull dimension $d-s$,
\item   $(0:_{\kk[\Delta]\{s\}}  \theta_{s+1})=H^0(\kk[\Delta]\{s\})$,
and
\item for all $i+s<d-1$, $$H^i_\m(\kk[\Delta]\{s+1\})_j \cong
  H^i_\m(\kk[\Delta]\{s\})_j \oplus
 H^{i+1}_\m(\kk[\Delta]\{s\})_{j-1}.$$
\end{enumerate} \end{lemma}

\begin{proof}
Part 1 is a consequence of Lemma \ref{Buchsbaum}(2)
above and Corollary 1.11 on p.~65 of \cite{StVo}; Part 2 follows
from Part 1 and Proposition 1.10 on p.~64 of \cite{StVo}; Part 3 is
immediately apparent from
 Part 1 and \cite[Lemma 8.2]{N}.
\end{proof}

\begin{lemma} \label{Hilb-cohomology}
For $\Delta$ and $\theta_1, \ldots, \theta_d$ as in Lemma
\ref{properties}, and for all $s\geq 1$ and $i+s < d$,
$$
\dim_\kk H^i_\m(\kk[\Delta]\{s\})_j=
\binom{s-1}{j-1}\beta_{i+j-1}(\Delta)+
\binom{s-1}{j}(\bK_{i+j}(\Delta)+\bC_{i+j-1}(\Delta)). $$
In particular, for $1\leq s\leq d-1$, $$\dim_\kk(0:_{\kk[\Delta]\{s\}}
\theta_{s+1})_j=
 \binom{s-1}{j-1}\beta_{j-1}(\Delta)+
\binom{s-1}{j}(\bK_{j}(\Delta)+\bC_{j-1}(\Delta)).$$ \end{lemma}

\begin{proof}
We use induction on $s$. For $s=1$ the statement reduces to
Lemma~\ref{Buchsbaum}(1). Part 3 of
Lemma~\ref{properties} together with the binomial identity
$\binom{a}{b}+\binom{a}{b-1}=\binom{a+1}{b}$ gives the inductive step.
The ``in particular''-part then follows from Lemma~\ref{properties}(2).
\end{proof}

We are now in a position to prove Theorem \ref{h'}.

\begin{proof}
Since for any l.s.o.p.~$\Theta=\{\theta_1,\ldots,\theta_d\}$ of
$\kk[\Delta]$, the linear span of $\Theta$ contains a 1-form with
all coefficients non-vanishing, we can assume w.l.o.g.~that $\theta_1$
is such a form. It is a well-known fact (see for instance
\cite[p.~150]{StVo}) that for a $(d-1)$-dimensional
complex $\Delta$ and an l.s.o.p.~$\Theta$,
$$
F(\kk(\Delta, \Theta), \lambda) = \sum_{i=0}^d  h_i(\Delta)\lambda^i +
 \sum_{s=1}^{d-1}\lambda(1-\lambda)^{d-s-1}\cdot
 F(0:_{\kk[\Delta]\{s\}}\theta_{s+1}, \lambda).
$$
Substituting in this equation
the ``in-particular''-part of Lemma \ref{Hilb-cohomology} yields
\begin{eqnarray*}
F(\kk(\Delta, \Theta), \lambda) & = & \sum_{i=0}^d  h_i(\Delta)\lambda^i
\\
& + &
 \mbox{\small $\sum_{s=1}^{d-1} (1-\lambda)^{d-s-1}
 \sum_{j=1}^s \lambda^{j+1}\left[\binom{s-1}{j-1}\beta_{j-1}(\Delta)
 +\binom{s-1}{j}(\bK_j(\Delta)+\bC_{j-1}(\Delta))\right]$.}
\end{eqnarray*}
Comparing the coefficients of $\lambda^i$ on both sides we obtain
\begin{eqnarray*}
h'_i(\Delta)=h_i(\Delta) &+&
 \sum_{j=1}^{i-1}(-1)^{i-j-1}
 \left(\sum_{s=j}^{d-1}
 \binom{d-s-1}{i-j-1}\binom{s-1}{j-1}\right)\beta_{j-1}(\Delta)\\
&+&  \sum_{j=1}^{i-1}(-1)^{i-j-1}
 \left(\sum_{s=j+1}^{d-1}
 \binom{d-s-1}{i-j-1}\binom{s-1}{j}\right)
    (\bK_{j}(\Delta)+\bC_{j-1}(\Delta)).
\end{eqnarray*}
Using that $\sum_{r=0}^l\binom{l-r}{i-j}\binom{r}{j}=\binom{l+1}{i+1}$
then implies
\begin{equation}  \label{bCK}
h'_i(\Delta)=h_i(\Delta)+
\sum_{j=1}^{i-1}(-1)^{i-j-1}\left[\binom{d-1}{i-1}\beta_{j-1}(\Delta)
+\binom{d-1}{i}(\bK_j(\Delta)+\bC_{j-1}(\Delta))\right].
\end{equation}
Finally, it follows from exactness of
$$
0 \to \ker(f^{j, \theta}) \to \bigoplus_{v\in V} H^j_{\{v\}}(\Delta)
 \stackrel{f^{j, \theta}}{\to} H^j_{\emptyset}(\Delta)
 \to \coker(f^{j,\theta}) \to 0
$$
that \begin{equation}  \label{bCtobK}
\bC_j(\Delta)=\bK_j(\Delta)+\beta_j(\Delta)-\sum_v \beta_{j-1}(\lk v).
\end{equation}
We infer the theorem by
substituting these expressions in Eq.~(\ref{bCK}) and using the identity
$\binom{d-1}{i-1}+\binom{d-1}{i}=\binom{d}{i}$.
\end{proof}

\begin{corollary} \label{socle}
Let $\Delta$ be a $(d-1)$-dimensional simplicial complex with
homologically isolated singularities, and let
$\Theta=\{\theta_1,\ldots,\theta_d\}$
be an l.s.o.p.~for $\kk[\Delta]$.
Then for all $i< d-1$,
 $$ \dim\soc(\kk(\Delta, \Theta))_i\geq
 \binom{d}{i}\beta_{i-1}(\Delta)+
 \binom{d-1}{i}(\bK_i(\Delta)+\bK_{i-1}(\Delta)-
            \sum_{v\in V}\beta_{i-2}(\lk v)),
$$
and hence $h'_i(\Delta)\geq \binom{d}{i}\beta_{i-1}(\Delta)+
 \binom{d-1}{i}(\bK_i(\Delta)+\bK_{i-1}(\Delta)-
            \sum_{v\in V}\beta_{i-2}(\lk v))$.
\end{corollary}
\begin{proof}
Since by Lemma \ref{Buchsbaum}(2), $\kk[\Delta]/(\theta_1)$ is a
Buchsbaum ring of Krull dimension $d-1$, Theorem 2.2 of \cite{NS-socle}
applies and gives
$$
\dim\soc(\kk(\Delta, \Theta))_i
\geq \sum_{j=0}^{d-2}\binom{d-1}{j}\dim_{\kk}
H^j_\m(\kk[\Delta]/(\theta_1))_{i-j}.
$$
By Lemma \ref{Buchsbaum}(1), for $i<d-1$ the right-hand side
of this inequality is equal to $\binom{d-1}{i-1}\beta_{i-1}(\Delta)+
 \binom{d-1}{i}(\bK_i(\Delta)+\bC_{i-1}(\Delta))$.
The result follows from Eq.~(\ref{bCtobK}).
\end{proof}

Note that if $\Delta$ is a Buchsbaum complex, then for $i<d-1$ and $v\in
V$, $\beta_{i-1}(\lk v)=\bK_i(\Delta)=0$.
Thus Theorem \ref{h'} reduces to Schenzel's theorem
and Corollary \ref{socle} to \cite[Theorem 3.4]{NS-socle}.
We also remark that Theorem \ref{h'} can be restated in terms of
singular homology.

\begin{theorem} \label{h'-version2}
Let $\Delta$ be a $(d-1)$-dimensional simplicial complex with
isolated singularities, $X=\|\Delta\|$ its geometric realization,
$\Sigma\subset X$ the set of singularities of $X$, and $\theta=\sum_{v\in
V} x_v$.   Then for all $j<d$, $
\bK^\theta_j(\Delta)+\bC^\theta_{j-1}(\Delta)
=\dim_\kk\tilde{H}_{j-1}(X-\Sigma; \kk)$.
Moreover, if $\Delta$ has homologically isolated singularities
and $\Theta$ is any l.s.o.p.~for $\kk[\Delta]$, then for $0\leq i\leq d$,
$$
h'_i(\Delta)=h_i(\Delta)+ \sum_{j=1}^{i-1}(-1)^{i-j-1} \left[
\binom{d-1}{i-1}\dim_\kk \tilde{H}_{j-1}(X; \kk)
+\binom{d-1}{i}\dim_\kk \tilde{H}_{j-1}(X-\Sigma; \kk)\right].
$$
\end{theorem}

\begin{proof} The second assertion follows the first one and
Eq.~(\ref{bCK}). To prove the first assertion consider
the following long cohomology sequence of the pair (where we omit $\kk$
to simplify the notation)
$$
H^{j-1}(X,X-\Sigma) \stackrel{\psi}{\to} \tilde{H}^{j-1}(X)
\to \tilde{H}^{j-1}(X-\Sigma)\to H^j(X,X-\Sigma)
\stackrel{\phi}{\to} \tilde{H}^j(X).
$$
It implies that $\dim_\kk \tilde{H}_{j-1}(X-\Sigma)
=\dim_\kk \tilde{H}^{j-1}(X-\Sigma)=\dim_\kk \coker(\psi)
+\dim_\kk \ker(\phi)$. Thus to derive the first assertion it suffices to
show that
 $\dim_\kk\coker(\psi)=\bC^\theta_{j-1}(\Delta)$ and
$\dim_\kk\ker(\phi)=\bK^\theta_{j}(\Delta)$. This is immediate
from the following series of canonical isomorphisms:
$$H^{j}(X,X-\Sigma) \cong \bigoplus_{v\in\Sigma} H^j(X, X-v)
\cong \bigoplus_{v\in V} H^j(X, X-v) \cong \bigoplus_{v\in V}
H^j_{\{v\}}(\Delta) \quad \mbox{for $j<d-1$}.
$$
The first isomorphism in this series follows, for instance,
by induction on the
size of $\Sigma$ using the relative form
of the Mayer-Vietoris sequence for the inductive step. The last
isomorphism is implied by the fact that $\|\cost v\|$ is a deformation
retract of $X-v$. The middle one is immediate from
the definition of a non-singular vertex: for such a vertex $v$ and
$j<d-1$, $H^{j}(X,X-v)\cong \tilde{H}^{j-1}(\lk v)=0$.
\end{proof}

\begin{remark}
When $X=\|\Delta\|$ is a {\it pseudomanifold} with isolated
singularities,
$\tilde{H}_{j-1}(X-\Sigma; \kk)$ coincides with
$H^{(\overline{0})}_{j-1}(X; \kk)$ --- the $(j-1)$-st intersection
homology of perversity $\overline{0}$. (Intersection homology was
introduced and studied
 in \cite{GorMac}.) Thus for pseudomanifolds
with homologically isolated singularities,
 $h'_i(\Delta)-h(\Delta)_i$ can be written
as a linear combination of the usual Betti numbers of $\Delta$
and the intersection Betti numbers of $\|\Delta\|$. \end{remark}

We now turn to numerical consequences of Theorem \ref{h'}.
Let $g_2(\Delta):=h_2(\Delta)-h_1(\Delta)$.
The celebrated Lower Bound Theorem \cite{Bar, Ka} asserts that if
$\Delta$ is a connected $\kk$-homology manifold of dimension $d-1 \geq
2$, then $g_2(\Delta)\geq 0$. Moreover, it was shown in \cite[Theorem
5.3]{NS-socle} that if $\Delta$ is a connected
orientable $\kk$-homology manifold of dimension $d-1\geq 3$, then
$g_2(\Delta)\geq\binom{d+1}{2}\beta_1(\Delta)$. Here we extend this
result to pseudomanifolds with homologically
isolated singularities. (For {\em normal} pseudomanifolds,
that is, pseudomanifolds all of whose faces of codimension at least two
have connected links, it is a result of Fogelsanger \cite{Fog}
that the lower bound $g_2\geq 0$ continues to hold.)

\begin{theorem} \label{LBT}
Let $\Delta$ be a $(d-1)$-dimensional pseudomanifold with homologically
isolated singularities. If $d\geq 5$, then
$$
g_2(\Delta)
\geq \binom{d+1}{2}\left[\beta_{d-2}(\Delta)-\beta_{d-1}(\Delta)+1\right]
+d \cdot\bK_{d-2}(\Delta)-
d\sum_{v\in V}\left[\beta_{d-3}(\lk v)-\beta_{d-2}(\lk v)+1\right].
$$
The same assertion also holds if $d=4$ and $\Delta$ has at most 5
singularities.
\end{theorem}

If $\Delta$ is a connected orientable $\kk$-homology manifold, then
$\beta_{d-1}(\Delta)=1$ and by Poincar\'e duality
$\beta_{d-2}(\Delta)=\beta_1(\Delta)$; also
for every vertex $v$, $\beta_{d-2}(\lk v)=1$ while $\beta_{d-3}(\lk
v)=0$.
Thus in this case Theorem \ref{LBT} is equivalent to the inequality
$g_2(\Delta)\geq\binom{d+1}{2}\beta_1(\Delta)$ mentioned above.  In
general the right-hand side is a topological invariant of $\Delta$ as all
of the nonsingular vertices contribute zero to the last term.

\begin{proof} First assume $d\geq 5$. Since $\Delta$ is a
$(d-1)$-dimensional
pseudomanifold with isolated singularities, it follows that the link of
every $(d-4)$-face $F$ is a 2-dimensional homology sphere,
and hence the boundary of a 3-dimensional simplicial polytope. Thus for a
generic $1$-form $\omega$, multiplication by $\omega$ considered as a
linear map from $\kk(\lk F)_1$ to $\kk(\lk F)_2$,
where $\kk(\lk F)$ is an Artinian reduction of $\kk[\lk F]$ by
a generic l.s.o.p., is an
isomorphism. (In characteristic 0 this result was proved in \cite{St80};
for a field of arbitrary characteristic it follows from Murai's paper
\cite{Mur} combined with Whiteley's proof that 2-dimensional spheres are
strongly edge decomposable \cite{Whit}.) The proof of Theorem 4.26 in
\cite{Sw09} then implies that if $\Theta$ is a generic
l.s.o.p.~of  $\kk[\Delta]$ and $\omega\in\kk[\Delta]$ is a generic
1-form, then the map
$$  \cdot\omega: \kk(\Delta, \Theta)_{d-2}\to \kk(\Delta, \Theta)_{d-1},
$$ is {\bf surjective}.
As the socle, $\soc(\kk(\Delta,\Theta))_{d-2}$, is in the kernel of this
map, we obtain that
\begin{equation} \label{surj}
 h'_{d-2}(\Delta) - \dim_{\kk}\soc(\kk(\Delta,\Theta))_{d-2} \geq
 h'_{d-1}(\Delta).
\end{equation}

If $d=4$ and $\Delta$ has at most 5 singular vertices, then we appeal to
Lemma \ref{onto} in Section~\ref{Examples} below to see that
Eq.~(\ref{surj}) holds in this case as well.

Since $\Delta$ has homologically isolated singularities,
Theorem~\ref{h'}  implies that
\begin{eqnarray}  \nonumber
h'_{d-1}(\Delta) &=&
h_{d-1}(\Delta)+
 d\left[\sum_{j=1}^{d-2}(-1)^{(d-1)-j-1}\beta_{j-1}(\Delta)\right]
+ \bK_{d-2}(\Delta) \\
\nonumber & \ & \qquad\qquad
-\sum_{v\in V}\sum_{j=1}^{d-3}(-1)^{(d-1)-j}\beta_{j-1}(\lk v)\\
\nonumber  &=& h_{d-1}(\Delta)+
 d\left[(-1)^{d-1}\tilde{\chi}(\Delta)-
\beta_{d-1}(\Delta)+\beta_{d-2}(\Delta)\right]
 + \bK_{d-2}(\Delta) \\
\nonumber & \ & \qquad\qquad
 -\sum_{v\in V}
 \left[(-1)^{d}\tilde{\chi}(\lk v)-
 \beta_{d-2}(\lk v)+\beta_{d-3}(\lk v)\right]\\
 &=& \label{d-1}
 h_1(\Delta)+
d\left[\beta_{d-2}(\Delta)-\beta_{d-1}(\Delta)+1\right]+\bK_{d-2}(\Delta)\\
\nonumber  & \ & \qquad\qquad   - \sum_{v\in V}\left[\beta_{d-3}(\lk
v)-\beta_{d-2}(\lk v)+1\right],
\end{eqnarray}
where the last step follows from Theorem \ref{gen-DS}.
A similar computation using Theorems \ref{h'} and \ref{gen-DS}, and
Corollary \ref{socle} gives
\begin{eqnarray} \nonumber
 h'_{d-2}(\Delta) &-& \dim_{\kk}\soc(\kk(\Delta,\Theta))_{d-2}\\
\label{d-2}
&\leq&
h_2-\binom{d}{2}\left[\beta_{d-2}(\Delta)-\beta_{d-1}(\Delta)+1\right]-
 (d-1)\bK_{d-2}(\Delta)\\
\nonumber  & \ & \qquad
 +(d-1) \sum_{v\in V}\left[\beta_{d-3}(\lk v)-
 \beta_{d-2}(\lk v)+1\right].
\end{eqnarray}
Substituting Eqs.~(\ref{d-1}) and (\ref{d-2}) in (\ref{surj}) and
combining like terms yields the result.
\end{proof}

\section{Depth $d-1$}  \label{depth}

The {\it depth} of $\Delta$ is the largest $j$ such that the
$(j-1)$-st skeleton of $\Delta$ is CM. For a space with isolated
singularities this is equivalent to $\tilde{H}_i(\Delta; \kk)=0$ for $i<j-1$ and for every vertex $v,\ \tilde{H}_i(\lk v; \kk) \cong H_{i+1}(\|\Delta\|, \|\Delta\|-v) = 0$ for $i<j-2$; in particular, for  a space with isolated singularities, having depth $j$ is a topological property. This notion coincides with the usual algebraic
notion of depth for $\field[\Delta]$, see \cite{Smi}.
Complexes of dimension $d-1$ and depth $d-1$ are also known in the
literature as {\em almost CM} complexes.

\begin{theorem} \label{high depth}
Let $\Delta$ be a $(d-1)$-dimensional complex with isolated singularities
with depth $d-1.$   Set $h'_i=\dim_{\kk}\field(\Delta,\Theta)_i.$  Then
for
sufficiently generic $\Theta$
$$\begin{array}{lcl}
h'_d & = & \beta_{d-1}(\Delta) \\
h'_{d-1} & = & h_{d-1}(\Delta) + \dim_\field
(\bK_{d-2}^{\theta_1}(\Delta) \cap \dots
\cap \bK_{d-2}^{\theta_d}(\Delta)) \\
h'_i & = & h_i(\Delta), \quad i<d-1.
\end{array}
$$
\end{theorem}

\begin{proof}
For {\it any} $(d-1)$-dimensional simplicial complex $h'_d= \beta_{d-1}$
for sufficiently generic $\Theta$, see \cite[Lemma 2.2(3)]{BN}.  Since
$\Delta$ has depth
$d-1$ the only possible nontrivial Betti numbers of $\Delta$ are
$\beta_{d-1}$ and $\beta_{d-2}$  and the only potential nonzero Betti
numbers for the link of every vertex $v$ are $\beta_{d-2} (\lk v)$ and
$\beta_{d-3}(\lk v).$  By Theorem \ref{Grabe}, $H^i_\mm(\field[\Delta])
= 0$ for $i < d-1$ and
$$\begin{array}{lcl}
H^{d-1}_\mm (\field[\Delta])_0 & \cong & H^{d-2}_\emptyset(\Delta) \\
H^{d-1}_\mm (\field[\Delta])_{-1} &\cong &\sum_v H^{d-2}_{\{v\}}(\Delta).
\end{array}$$

Since $H^0_\mm(\kk[\Delta]) = 0$, a prime avoidance argument (see, for
instance,
\cite[Chapter 3]{Eis}) implies the existence of $\theta_1 \in
\field[\Delta]_1$ such
that multiplication by $\theta_1$ on $\field[\Delta]$ is an
injection.  Now apply the induced long exact sequence for local
cohomology with respect to the short exact sequence
$$ 0 \to \field[\Delta] \stackrel{\cdot \theta_1}{\to} \field[\Delta] \to
\field[\Delta]/(\theta_1) \to 0$$ to see that
$H^i_\mm(\field[\Delta]/(\theta_1)) = 0$ for $i < d-2.$ In
addition,   the module structure of $H^{d-1}_\mm(\field[\Delta])$
implies that
$$\begin{array}{lcl}
H^{d-2}_\mm (\field[\Delta]\{1\})_1 & \cong & H^{d-2}_\emptyset(\Delta) \\
H^{d-2}_\mm (\field[\Delta]\{1\})_0 &\cong &\bK_{d-2}^{\theta_1}(\Delta), \\
H^{d-2}_\mm (\field[\Delta]\{1\})_i &\cong & 0,\quad i \neq 0,1.
\end{array}$$

We can repeat this process until we produce $\theta_1, \dots,
\theta_{d-1}$ such that multiplication $\cdot \theta_i:
\field[\Delta]\{i-1\} \to \field[\Delta]\{i-1\}$ is an injection for each
$1 \le i \le d-1$ and $$\begin{array}{lcl}
H^0_\mm (\field[\Delta]\{d-1\})_{d-1} & \cong & H^{d-2}_\emptyset(\Delta) \\
H^0_\mm (\field[\Delta]\{d-1\})_{d-2} &\cong &\bK_{d-2}^{\theta_1}(\Delta) \cap
\dots
\cap \bK_{d-2}^{\theta_{d-1}}(\Delta), \\
H^0_\mm (\field[\Delta]\{d-1\})_i &\cong & 0,\quad i \neq d-1,d-2.
\end{array}$$

 One last prime avoidance argument to find $\theta_d$ such that the
kernel of multiplication on $\field[\Delta]\{d-1\}_{d-2}$ is isomorphic
to $\bK_{d-2}^{\theta_1}(\Delta) \cap \dots \cap
\bK_{d-2}^{\theta_d}(\Delta)$ finishes the proof.
\end{proof}

 Any  pure two-dimensional simplicial complex has isolated
singularities.  Indeed, the link of any edge is a nonempty collection of
points and hence Cohen-Macaulay.  As $h'_3 = \beta_2, h'_1 = h_1$, and
$h'_0 = h_0 =1$, it only remains to compute $h'_2.$

 In \cite{Lee} Lee established that $h'_2$ is
the dimension of the space of stresses
of a generic embedding of the one-skeleton of
$\Delta$ into the plane. (Various defenitions pertaining to rigidity
theory such as those of
stresses, infinitesimal motions, etc.~can be found in \cite{GrSerSer}.)
As an algorithm for the latter for any graph,
in particular for the one-skeleton of a simplicial complex, had already
been determined, see  \cite[Theorem 4.4.3]{GrSerSer} (it is a
consequence
of Laman's theorem), a {\it combinatorial} formula for
$h'_2$ already exists.   The following corollary states $h'_2$ using the
{\it topology} and $f$-vector $\Delta.$

\begin{corollary}
Let $\Delta$ be a connected two-dimensional pure complex.  Then for
sufficiently generic $\Theta,$
$$
h'_2  =  h_2 + \dim_\field (\bK_1^{\theta_1}(\Delta) \cap
\bK_1^{\theta_2}(\Delta) \cap \bK_1^{\theta_3}(\Delta)).$$
\end{corollary}

\begin{proof}
The $1$-skeleton of $\Delta$ is CM if and only if $\Delta$ is connected.
\end{proof}

What about two-dimensional complexes that are not connected?  One can
either appeal to the rigidity formula and the fact that each extra
component introduces three new infinitesimal  motions or a local
cohomology argument similar to the one in Theorem  \ref{high depth} to
find that the formula is
\begin{equation} \label{2d}
h'_2= h_2 + 3 \beta_0(\Delta) +  \dim_\field (\bK_1^{\theta_1}(\Delta)
\cap
\bK_1^{\theta_2}(\Delta) \cap \bK_1^{\theta_3}(\Delta)).
\end{equation}

\section{Examples} \label{Examples}

In this section we use methods and results developed in the paper so far
to provide a complete characterization of the $f$-vectors of arbitrary
simplicial triangulations of certain 3-dimensional pseudomanifolds
denoted
by $\|N_1\|$, $\|N_3\|$, and $\|N_4\|$ in \cite[Example 4]{DaNi}.
We start by recalling from \cite{DaNi} the necessary information
on the topology of $X_i:=\|N_i\|$, $i=1,3,4$.

The pseudomanifold $X_1$ is the only normal 3-pseudomanifold with 8
singularities all of whose links are the torus that has an 8-vertex
triangulation; its homology groups (over $\Z$) are
$\tilde{H}_0=\tilde{H}_1=0$, $\tilde{H}_2=\Z^8$, and $\tilde{H}_3=\Z$.
Therefore, for any triangulation $\Delta$
of $X_1$, \begin{equation} \label{X1}
h_4(\Delta)=-\tilde{\chi}(\Delta)=-7, \quad \sum_v (1-\tilde{\chi}(\lk
v))=16, \quad
\mbox{and so } h_3(\Delta)-h_1(\Delta)=16,
\end{equation}
where the last equality is by Theorem \ref{gen-DS}.

The pseudomanifold $X_3$ has 5 singularities; the links of the 5 singular
vertices in $N_3$ have different topological types: four of them are
projective planes, while the fifth one is a 2-dimensional torus. The
homology
groups of $X_3$ are $\tilde{H}_0=\tilde{H}_1=\tilde{H}_3=0$ and
$\tilde{H}_2=\Z^2\oplus\Z/2\Z$. Thus for any triangulation $\Delta$
of $X_3$, \begin{equation}  \label{X3}
h_4(\Delta)=-\tilde{\chi}(X_3)=-2, \quad \sum_v (1-\tilde{\chi}(\lk
v))=6,
\mbox{ hence }h_3(\Delta)-h_1(\Delta)=6.
\end{equation}

The pseudomanifold $X_4$ is $H\cup C(\partial H)$, where $H$ is a solid
$3$-dimensional torus and $C(\partial H)$ is the cone over the boundary
of $H$.
Hence $X_4$ has only one isolated singularity,
and for any triangulation $\Delta$ of $X_4$, $\bK_2(\Delta)=1$.
In addition,
$$ h_4(\Delta)=-\tilde{\chi}(\Delta)=0,
 \quad \sum_v (1-\tilde{\chi}(\lk v))=2, \quad
\mbox{and so } h_3(\Delta)-h_1(\Delta)=2.
$$

As the above discussion shows, to characterize the $h$-vectors
(equivalently, the $f$-vectors) of arbitrary simplicial  triangulations
of $X_i$ ($i=1,3,4$) it is enough to characterize all possible pairs
$(h_1(\Delta), h_2(\Delta))$, where $\|\Delta\|=X_i$. This, in turn, is
equivalent to characterizing the set
 $$\{(g_1(\Delta), g_2(\Delta)) : \|\Delta\|=X_i\}, \quad \mbox{where }
 g_j(\Delta):=h_j(\Delta)-h_{j-1}(\Delta).$$
To this end, we have the following result.

\begin{theorem} \label{N1,3,4} For $i=1,3,4$ and for a pair $(g_1,
g_2)\in\Z^2$, the following are equivalent:
\begin{enumerate}
\item[(i)] There exists a simplicial complex $\Delta$, $\|\Delta\|=X_i$
such that
$(g_1(\Delta), g_2(\Delta))=(g_1,g_2)$;
\item[(ii)] $6\leq g_2 \leq \binom{g_1+1}{2}$. (In particular, $g_1\geq
3$.)
\end{enumerate}
\end{theorem}

\begin{proof}
To verify that (ii) implies (i), start with the triangulation $N_i$ of
$X_i$ described in \cite{DaNi} (it satisfies $g_1(N_i)=3$ and
$g_2(N_i)=6$)
and follow the proof of Lemma 7.3 in Walkup's paper~\cite{Wal}.  The
triangulations $N_i$ (taken from \cite{DaNi}) and the ``simple
$3$-trees'' required to invoke this argument are listed at the end of
this section.

To prove that (i) implies (ii), fix a field $\kk$ of characteristic zero.
Consider a triangulation $\Delta$ of $X_i$, its face ring $\kk[\Delta]$,
a generic l.s.o.p.~for  $\kk[\Delta]$, $\Theta=\{\theta_1, \ldots,
\theta_4\}$,
and one additional generic linear form $\omega$. Since $\Delta$ is a
normal 3-dimensional pseudomanifold, it follows from \cite{Fog}
that $\Delta$ is generically 4-rigid in the graph-theoretic sense (see
\cite{GrSerSer} or \cite{Ka} for appropriate definitions). This result
combined with \cite[Theorem 10]{Lee}
yields that
$\dim_\kk(\Delta, \Theta)_1=h_1(\Delta)$, $\dim_\kk(\Delta,
\Theta)_2=h_2(\Delta)$
and the map $\cdot\omega: \kk(\Delta, \Theta)_1 \to \kk(\Delta,
\Theta)_2$ is injective.
Thus for $j\leq 2$, $\dim_\kk [\kk(\Delta,
\Theta)/\omega]_j=g_j(\Delta)$, and the inequality
$g_2(\Delta)\leq\binom{g_1(\Delta)+1}{2}$ follows from Macaulay's theorem
(see \cite[pp.~56-57]{St96}).

It only remains to show that $g_2(\Delta)\geq 6$. If $\Delta$ is a
triangulation of $X_4$, then $\Delta$ has one singular vertex, and the
result follows from Theorem \ref{LBT}  according to which
$g_2(\Delta) \geq \binom{5}{2}[1-1+1]+4\cdot 1-4[2-1+1]=6$.

If $\Delta$ is a triangulation of $X_3$, then as a $\kk$-space (rather
than $\kk$-pseudomanifold) with isolated singularities, $X_3$ has only
one singular point. (This is because over a field of characteristic 0,
the projective plane is acyclic, and hence CM.)
Theorem \ref{h'} then implies that $\dim_\kk(\Delta,
\Theta)_3=h_3(\Delta)$.
Considered as a $\kk$-pseudomanifold, $\Delta$ has only 5 singular
points, and we infer from Lemma \ref{onto} below  that the map
$\cdot\omega: \kk(\Delta, \Theta)_2 \to \kk(\Delta, \Theta)_3$ is onto,
and hence $h_2(\Delta)\geq h_3(\Delta)$.
Thus $$g_2(\Delta)=h_2(\Delta)-h_1(\Delta)\geq h_3(\Delta)-h_1(\Delta)=6
\quad \mbox{ by Eq.~(\ref{X3})}.
$$

Finally, let $\Delta$ be a triangulation of $X_1$.  Since the given
triangulation $N_1$ of $X_1$ in \cite{DaNi} contains the complete
$2$-skeleton on eight vertices its depth is $3.$ By Theorem \ref{high
depth},
$\dim_\kk \kk(\Delta, \Theta)_3=h_3(\Delta)$ plus a term which only
depends on the topology of $X_1.$  However, as all of the missing faces
of $N_1$ are
three-dimensional, $\dim_\kk \field(N_1, \Theta)_3=h_3$ for this
triangulation,
and thus for all others. This implies that $\dim_\kk [\kk(\Delta,
\Theta)/\omega]_3 \geq
g_3(\Delta)$, and hence that $(1,g_1(\Delta), g_2(\Delta), \max\{0,
g_3(\Delta)\})$ is an M-sequence.
Thus according to Macaulay's theorem,
if $g_2(\Delta)<6=\binom{4}{2}$, then $g_3(\Delta)<\binom{5}{3}=10$, and
hence $h_3(\Delta)-h_1(\Delta)=g_2(\Delta)+g_3(\Delta)<6+10=16$,
contradicting Eq.(\ref{X1}). Therefore, $g_2(\Delta)\geq 6$ in this case
as well.
The theorem follows.
\end{proof}

\begin{lemma} \label{onto}
Let $\Delta$ be a $3$-dimensional pseudomanifold with isolated
singularities,
$\Theta=\{\theta_1,\ldots, \theta_4\}$
 a generic l.s.o.p.~of $\kk[\Delta]$, and $\omega\in \kk[\Delta]_1$ a
generic 1-form. If $\Delta$ has at most $5$ singularities then the map
$\cdot\omega : \kk(\Delta,\Theta)_2 \to \kk(\Delta, \Theta)_3$ is
surjective.
\end{lemma}

\begin{proof}
Let $U \subseteq V$ be the subset of singular vertices of $\Delta$ and
let $\Delta_U$ be the induced subcomplex.  Consider the following
commutative diagram.
$$\begin{array}{ccccccccc}
0 & \to & I_2 & \to & \field(\Delta, \Theta)_2 & \to & \field(\Delta_U,
\Theta)_2 & \to & 0 \\
 & &\ \  \downarrow \cdot \omega & & \downarrow \cdot \omega & &
\downarrow \cdot \omega & & \\
 0 & \to & I_3 & \to & \field(\Delta, \Theta)_3 & \to & \field(\Delta_U,
\Theta)_3 & \to & 0
 \end{array}$$
 By the snake lemma it is sufficient to show that the two outside maps
are surjective for generic choices of $\Theta$ and $\omega.$  As $|U|
\le 5$, $\field(\Delta_U, \Theta)_3 = 0$ unless $|U|=5$
and $\Delta_U$ contains the
entire $2$-skeleton of the 4-simplex with vertex set ${U}.$ However, if
this is the case, then the
right-hand map is the same as for the boundary of the $4$-simplex, where
it is known to be surjective.  For the left-hand map, first we observe
that $I_3$ is generated by all monomials that contain at least one
nonsingular vertex.  Now we follow the proof of \cite[Theorem
4.26]{Sw09} to see that each of these monomials is in the image.
\end{proof}

$$\begin{array}{ccc}
\mbox{Triangulation} & \mbox{Facets} & \mbox{Simple $3$-tree} \\
 & & \\
 \mbox{N1} & 1248, 1268, 1348, 1378, 1568, 1578, & \quad
1248,1268,1348,1378,1568\\
&  2358, 2378, 2458, 2678, 3468, 3568, & \\
& 4578, 4678, 1247, 1257, 1367, 1467, & \\
& 2347, 2567, 3457, 3567, 1236, 2346, & \\
& 1345, 1235, 1456, 2456 & \\
& & \\
N3 & 1248, 1268, 1348, 1378, 1568, 1578 & \quad 1248, 1268, 1348, 1378,
1568 \\
&  2358, 2378, 2458, 2678, 3468, 3568, & \\
& 4578, 4678, 1234, 2347, 2456, 2467, & \\
& 3456, 3457, 1235, 1256, 1357 & \\
& & \\
N4 & 1248, 1268, 1348, 1378, 1568, 1578, & \quad
1248, 1268, 1348, 1378, 1568
\\
& 2358, 2378, 2458, 2678, 3468, 3568, & \\
& 4578, 4678, 1245, 1256, 2356, 2367, & \\
&  3467, 1347, 1457 &
\end{array}$$

\section{PL homeomorphic pseudomanifolds}

As Theorem \ref{h'-version2} shows, for a simplicial complex $\Delta$
with homologically isolated singularities, the quantities
 $\dim_\kk \kk(\Delta,\Theta)_i -h_i(\Delta)$, for
$0\leq i \leq \dim\Delta+1$,
are independent of the choice of an l.s.o.p.~$\Theta$; moreover,
they depend only on the homeomorphism type of $\|\Delta\|$.
While for an arbitrary complex $\Delta$,
the Hilbert function of $\kk(\Delta,\Theta)$ does depend on $\Theta$,
for sufficiently generic linear systems of parameters it is a constant,
and we define
$$h'_i(\Delta):=\dim_\kk \kk(\Delta,\Theta)_i, \quad \mbox{where $\Theta$
 is a generic l.s.o.p.~for } \kk[\Delta].
$$
Since in this section we work only with generic linear systems of
parameters, to simplify the notation we write $\kk(\Delta)$ instead of
$\kk(\Delta,\Theta)$.

A natural question raised in \cite{MNS} is whether for an arbitrary
$(d-1)$-dimensional complex $\Delta$,
$h'_i(\Delta)-h_i(\Delta)$ is determined by the homeomorphism type of
$\|\Delta\|$. As noted earlier, for $i=d$, the answer is yes as
$h'_d(\Delta)=\beta_{d-1}(\Delta)$, \cite[Lemma 2.2(3)]{BN}. The main
result of this section is that
for spaces with isolated singularities, $h'_i-h_i$ is determined
by the PL-homeomorphism type for all $i$.

\begin{theorem}  \label{PL-homeo}
If two simplicial complexes with isolated
singularities,  $\Delta$ and $\Gamma$, are PL-homeomorphic, then
$h'_i(\Gamma)-h_i(\Gamma)=h'_i(\Delta)-h_i(\Delta)$ for all $i$.
\end{theorem}

Since PL-homeomorphic complexes always have a common subdivision,
and since the link of any face $F$, $\dim F>0$, of a complex with
isolated singularities is CM, Theorem~\ref{PL-homeo} is a corollary to
the
following result. We recall that the {\em stellar subdivision} of a
simplicial complex $\Delta$ at a face $F$
(where $F\in \Delta$ and $\dim F>0$)
is the simplicial complex obtained from $\Delta$ by removing all faces containing $F$ and adding a new vertex $\rho$ as well as all sets of the form $\tau\cup\{\rho\}$ where $\tau$ does not contain $F$ but $\tau\cup F\in\Delta$.

\begin{theorem}   \label{main}
Let $\Delta$ be a $(d-1)$-dimensional simplicial complex, $F$ a face of
$\Delta$ of positive dimension, and $\Gamma$ the stellar subdivision of
$\Delta$ at $F$. If the link of $F$ in $\Delta$, $\lk F$, is
Cohen--Macaulay of dimension $d-1-|F|$, then
$h'_i(\Gamma)-h_i(\Gamma)=h'_i(\Delta)-h_i(\Delta)$ for all $i$.
\end{theorem}

\begin{proof}
 Our proof is motivated by the ideas from \cite{KuNe}.
Let $V$ be the vertex set of $\Gamma$ and  $V-\{\rho\}$ the vertex set
of $\Delta$. We write $\overline{F}$
to denote the simplex $F$ with all its faces. Consider the following
subcomplexes:
 \begin{eqnarray*}
 L&=&\overline{F}\ast\lk F
 \mbox{ --- the closed star of $F$ in $\Delta$,} \\
 R&=&\{G\in\Delta : G\not\supseteq F\}
   \mbox{ --- the costar of $F$ in $\Delta$,}\\
 L'&=&\overline{\rho}\ast \partial \overline{F} \ast \lk F
 \mbox{ --- the join of the boundary complex of $F$
   with $\rho$ and the link of $F$.}
 \end{eqnarray*}
Then $\Delta=L\cup R$, $\Gamma=L'\cup R$, and $L\cap R=\partial
\overline{F}\ast\lk F=L'\cap R$  --- denote
this common intersection by $T$. Therefore, we have the following short
exact sequences of $S$-modules (recall that $S=\kk[x_v : v\in V]$),
where the maps $\phi,\psi, \phi',\psi'$ are induced by inclusions of the
corresponding subcomplexes:
\begin{eqnarray*}
0\longrightarrow \kk[\Delta]
 \stackrel{\psi}{\longrightarrow} \kk[L]\oplus\kk[R]
\stackrel{\phi}{\longrightarrow} \kk[T] \longrightarrow 0, \\
0\longrightarrow \kk[\Gamma]
 \stackrel{\psi'}{\longrightarrow} \kk[L']\oplus\kk[R]
\stackrel{\phi'}{\longrightarrow} \kk[T] \longrightarrow 0.
\end{eqnarray*}

Let $\Theta=(\theta_1, \ldots, \theta_d)$ be an ideal of $S$ generated by
$d$ generic linear forms. In particular, we can assume that $\{\theta_1,
\ldots, \theta_d\}$ is an l.s.o.p.~for each of the modules
$\kk[\Delta]$, $\kk[\Gamma]$, $\kk[L]$, and $\kk[L']$. (Since
$\dim\Delta=\dim\Gamma=\dim L=\dim L'=d-1$, all four of these modules
have Krull dimension $d$.
Note however that $T$ and $R$ may have a strictly smaller dimension.)

We now tensor the above two sequences with $S/\Theta$ (over $S$).
As $\lk F$ is CM of dimension $d-1-|F|$, it follows that
$L=\overline{F}\ast \lk F$ and $L'=\overline{\rho}\ast \partial
\overline{F}\ast \lk F$ are both
CM of dimension $d-1$, and so $\kk[L]$ and $\kk[L']$ are free
$\kk[\theta_1, \ldots, \theta_d]$-modules. The fact that for any
$S$-module $M$,
$\tor_1^S(M, S/\Theta)=\ker (M\otimes_s \Theta S \rightarrow M)$,
then implies that $$\tor_1^S(\kk[L], S/\Theta)=\tor_1^S(\kk[L'],
S/\Theta)=0.$$ Thus by naturality of the long exact $\tor$ sequence, we
obtain the following commutative diagram:
\[\minCDarrowwidth16pt\begin{CD}
\tor_1^S(\kk[R], S/\Theta) @>{\phi_\star}>>  \tor_1^S(\kk[T], S/\Theta)
 @>{\delta}>> \kk(\Delta)
 @>{\psi_\star}>>  \kk(L)\oplus \kk[R]/\Theta @>{\phi_\star}>>
\kk[T]/\Theta @>>> 0 \\
 @A{\id}AA  @A{\id}AA  @.  @. @. \\
\tor_1^S(\kk[R], S/\Theta) @>{\phi'_\star}>>  \tor_1^S(\kk[T], S/\Theta)
 @>{\delta'}>> \kk(\Gamma)
 @>{\psi'_\star}>> \kk(L') \oplus \kk[R]/\Theta @>{\phi'_\star}>>
\kk[T]/\Theta @>>> 0, \\
\end{CD}
\]
where $\id$ is the identity map, and $\delta$ and $\delta'$ are
connecting homomorphisms. Hence for all $i$,
\[
\dim_{\kk} \kk(\Gamma)_i - (\dim_{\kk} (\kk[R]/\Theta)_i+\dim_{\kk}
\kk(L')_i)=
\dim_{\kk} \kk(\Delta)_i - (\dim_{\kk} (\kk[R]/\Theta)_i+\dim_{\kk}
\kk(L)_i), \]
yielding that \[
h'_i(\Gamma)-h'_i(\Delta)=h'_i(L')-h'_i(L)=h_i(L')-h_i(L)=
h_i(\Gamma)-h_i(\Delta) \quad \forall i,
\]
as required.
The second step in this computation holds since $L$ and $L'$ are CM, and
thus $h'(L)=h(L)$
and $h'(L')=h(L')$; the last step follows from the fact that
$f_j(\Delta)=f_j(L)+f_j(R)-f_j(T)$
and  $f_j(\Gamma)=f_j(L')+f_j(R)-f_j(T)$ for all $j$, so that
$f_j(L')-f_j(L)=f_j(\Gamma)-f_j(\Delta)$ for all $j$, and hence
also $h_i(L')-h_i(L)=h_i(\Gamma)-h_i(\Delta)$ for all $i$.
\end{proof}

\small{

}


\begin{thebibliography}{99}

\bibitem{BN} E.~Babson and I.~Novik, \emph{Face numbers and nongeneric
initial ideals},
 Electron.~J.~Combin.~\textbf{11} (2004/06), no. 2,
 Research Paper \#25, 23 pp.


\bibitem{Bar}
 D.~Barnette, \emph{A proof of the lower bound conjecture for convex
polytopes}, Pacific J.~Math.~\textbf{46} (1973), 349--354.


 \bibitem{CY}
 B.~Chen and M.~Yan, \emph{Eulerian stratification of polyhedgra},
Adv.~Appl.~Math.~\textbf{21} (1998), 22--57.


\bibitem{DaNi}
 B.~Datta and N.~Nilakantan, \emph{Three-dimensional pseudomanifolds on
eight vertices}, Int. J. Math. Math. Sci., 2008, Art.~ID 254637, 21 pp.

\bibitem{Eis}
D.~Eisenbud,
\emph{Commutative Algebra with a View Toward Algebraic Geometry},
Springer-Verlag, New York, 1995.

\bibitem{Fog} A. Fogelsanger, \emph{The generic rigidity of minimal
cycles}, PhD. Dissertation, Cornell University, 1988.

\bibitem{GorMac} M.~Goresky and R.~MacPherson, \emph{Intersection
homology theory}, Topology \textbf{19} (1980), 135--162.


\bibitem{Gr} H.-G.~Gr\"abe, \emph{The canonical module of a
Stanley-Reisner ring}, J.~Algebra \textbf{86} (1984), 272--281.

\bibitem{GrSerSer}
J.~Graver, B.~Servatius, and H.~Servatius,
\emph{Combinatorial Rigidity}, Graduate Studies in Mathematics, Vol.~2,
 American Mathematical Society, Providence, RI, 1993.

\bibitem{Ka} G.~Kalai, \emph{Rigidity and the lower bound theorem I},
Invent.~Math.~\textbf{88} (1987), 125--151.


\bibitem{Kl64}
V.~Klee,
\emph{A combinatorial analogue of Poincar\'e's duality theorem},
 Canad.~J.~Math.~\textbf{16} (1964), 517--531.


\bibitem{KuNe}
M.~Kubitzke and E.~Nevo,
\emph{The Lefschetz property for barycentric subdivisions of shellable
complexes}, Trans.~Amer.~Math.~Soc.~\textbf{361} (2009), 6151--6163.


\bibitem{Lee} C.~W.~Lee, \emph{Generalized stress and motions}, in
``Polytopes: abstract, convex and computational'' (Scarborough, ON,
1993), 249--271, NATO Adv.~Sci.~Inst.~Ser.~C Math.~Phys.~Sci., 440,
Kluwer Acad. Publ., Dordrecht, 1994.



\bibitem{MNS}
E.~Miller, I.~Novik, and E.~Swartz, \emph{Face rings of complexes with
singularities}, arXiv:1001.2812.


\bibitem{Mur}
S.~Murai, \emph{Algebraic shifting of strongly edge decomposable
spheres}, J.~Combin.~Theory Ser.~A \textbf{117} (2010), no. 1, 1--16.



\bibitem{N}
I.~Novik, \emph{Upper bound theorems for homology manifolds}, Israel J.
Math.~\textbf{108} (1998), 45--82.


\bibitem{NS-socle}
I.~Novik and E.~Swartz,
\emph{Socles of Buchsbaum modules, complexes and posets}, Advances in
Math.~\textbf{222} (2009), 2059--2084.


\bibitem{Reis}
G.~Reisner, \emph{Cohen-Macaulay quotients of polynomial rings}, Advances
in Math.~\textbf{21} (1976), 30--49.



\bibitem{Sch}
P. Schenzel, \emph{On the number of faces of simplicial complexes and the
purity of Frobenius},
Math.~Z.~\textbf{178} (1981), 125--142.

\bibitem{Smi} D.~E.~Smith, \emph{On the Cohen-Macaulay property in
commutative algebra and
simplicial topology}, Pacific J.~Math.~\textbf{141} (1990), 165--196.

\bibitem{St77} R.~Stanley, \emph{Cohen--Macaulay complexes}, in:
M.~Aigner
(Ed.), Higher Combinatorics, Reidel,
Dordrecht and Boston, 1977, pp. 51--62.


\bibitem{St80} R.~Stanley, \emph{The number of faces of a simplicial
convex polytope}, Advances
in Math.~\textbf{35} (1980), 236--238.


\bibitem{St96}
R.~Stanley,
\emph{Combinatorics and Commutative Algebra},
Birkh\"auser, 1996.


\bibitem{StVo}
J.~St{\"{u}}ckrad and W.~Vogel,
\emph{Buchsbaum Rings and Applications},
Springer-Verlag, Berlin, 1986.


\bibitem{Sw04}
E.~Swartz,
\emph{Lower bounds for $h$-vectors of $k$-CM, independence, and broken
circuit complexes}, SIAM J.~Discrete Math.~\textbf{18} (2004/05), no.~3,
647--661.


\bibitem{Sw09}
E.~Swartz,
\emph{Face enumeration: from spheres to manifolds},
 J.~Europ.~Math.~Soc.~\textbf{11} (2009), 449--485.


\bibitem{Wal}
D.~W.~Walkup,
\emph{The lower bound conjecture for $3$- and $4$-manifolds},
 Acta Math.~\textbf{125} (1970), 75--107.

\bibitem{Wall}
C.~T.~C.~Wall,
\emph{Arithmetic invariants of subdivision of complexes},
Canad.~J.~Math.~\textbf{18} (1966), 92--96.

\bibitem{Whit} W.~Whiteley, \emph{Vertex splitting in isostatic
frameworks}, Structural Topology \textbf{16} (1990), 23--30.


\end{thebibliography}
\end{document}